\documentclass[11pt, twosided]{amsart}
\usepackage{amsbsy,amssymb,pstricks,pst-node,mathrsfs,MnSymbol,stmaryrd}
\usepackage{subfigure}
\usepackage[left]{lineno}
\usepackage[inline]{enumitem}
\usepackage[left=3.5cm,right=3cm,top=2cm,bottom=1cm]{geometry}
\usepackage{tikz}
\usetikzlibrary{positioning}
\theoremstyle{definition}

\newtheorem{theorem}{Theorem}[section]

\newtheorem{lemma}[theorem]{Lemma}
\newtheorem{remark}[theorem]{Remark}
\newtheorem{proposition}[theorem]{Proposition}

\newtheorem{corollary}[theorem]{Corollary}

\newtheorem{definition}[theorem]{Definition}
\newtheorem{example}[theorem]{Example}

\begin{document}
	
	
	\title{On Isomorphism theorem of the Comparability Graph of Lattices}
	
	\maketitle \markboth{Rahul Jejurkar and Vinayak Joshi}{Isomorphism and Property Transfer in Lattices}
	
	\begin{center} \begin{large} Rahul Jejurkar$^{*}$  and Vinayak Joshi$^{**}$ \end{large}\\
		
		\begin{small}\vskip.1in\emph{$^{*}$School of Technology, Managment and Engineering,\\ SVKM NMIMS Global University, Dhule - 424001, Maharashtra, India
		\\$^{**}$Department of Mathematics, Savitribai Phule Pune University,\\ Pune - 411007, Maharashtra, India}\\
		E-mail: \texttt{rahuljejurkar1310@gmail.com, vinayakjoshi111@yahoo.com}\end{small}\end{center}
	
	\vskip.2in
	
	\begin{abstract}  
		In recent years, researchers have actively contributed to the field of graphs associated with algebraic structures and ordered structures. It is a fundamental question to ask whether we infer algebraic or ordered structure from associated graphs and vice versa. In this paper, we gave characterizations about comparability	graphs and associated lattices. In particular, we determined some properties of lattices that are preserved under the graph isomorphism. We have also provided a technique to construct non-isomorphic lattices with isomorphic comparability graphs. Also, we find two classes of lattices in which the graph isomorphism gives the lattice isomorphism.
		
	\end{abstract}
	
	\vskip.2in
	
	\noindent\begin{Small}\textbf{Mathematics Subject Classification (2020)}: 05C60, 06C05.
		
		\noindent 
		\textbf{Keywords}: Comparability graph, isomorphism, modular, distributive, atomistic, dual atomistic, complemented. \end{Small}
	
	\vskip.25in
	
	\baselineskip 16truept

\section{Introduction}
		
	There are two natural approaches relating a graph with an ordered structure, one by using the covering relation and the other by using the comparability relation. The resultant graphs are known as the covering graph $Cov(P)$ and the comparability graph $Com(P)$ associated with a poset $P$. The complement of the comparability graph is known as the incomparability graph. Recently, as like covering graph and comparability graph, Bre\v{s}ar, Changat et al. in  \cite{changat1}, introduced the cover-incomparability graph, abbreviated as a C-I graph,   associated with a poset, whose
	edge-set is the union of edge-sets of the incomparability and the cover graph of that poset. The C-I  graphs also form an interesting class of graphs from posets. 
	
	
	It is an important problem to determine all posets whose covering graph or comparability graph is given. Generally, the covering or comparability graph does not determine the associated poset up to isomorphism. This suggests that there is a loss of information in a graph theoretic sense. On the other hand, there are some specific classes of posets whose covering graph $Cov(P)$ and comparability graph $Com(P)$ still provides important information about the underline poset $P$.  
	
	In comparison to the comparability graph, a lot of progress has been accomplished in the case of covering graphs. Ward \cite{ward}, and Dilworth  \cite{dilworth} deduced some interesting conclusions in this area in the last century. In \cite{ward}, Ward proved that a modular lattice of finite length is distributive if and only if its covering graph contains no subgraph isomorphic to $ K_{2,3} $ (the covering graph of $ M_3 $). In \cite{dilworth}, Dilworth proved that the covering graph of a modular lattice of finite length and of breadth $ n $ contains a subgraph isomorphic to the $ n $-dimensional hypercube, i.e., $\mathbf{2}^n$.
	
	Ore \cite{Ore} essentially posed the problem: Characterize those finite unoriented graphs $G$ for which there exists a poset $P$ such that $G \cong Cov(P)$. This problem is one of the major open problems of order theory. On similar lines, Birkhoff \cite{Birkhoff} asked for necessary and sufficient conditions on lattices $L_1$  and $L_2$ such that their covering graphs are isomorphic. To tackle this problem, Jakubik \cite{jak} proved that for modular lattices, all graph isomorphisms are given by certain direct-product decomposition. In \cite{jak}, Jakubik essentially proved the following results.
	
	\begin{theorem}[{\cite[Theorem 2]{jak}}]{\label{jak1}}
		Let $L_1$ and $L_2$ be   modular lattices of finite length such that $Cov(L_1) \cong Cov(L_2)$. Then there are lattices $ A, B $ such that $L_1 \cong A \times B $ and $L_2 \cong A^d \times B$.
	\end{theorem}
	
	\begin{theorem}[{\cite[Theorem 1, Theorem 3]{jak}}]{\label{jak2}}
		Let $L_1$ and $L_2$ be lattices of finite length such that $Cov(L_1) \cong Cov(L_2)$. Then $L_1$ is modular (distributive) if and only if  $L_2$ is modular (distributive).
	\end{theorem} 
	
	Along similar lines, Duffus and Rival \cite{Duff} modified some of those results and solved the problem for those lattices of finite length that are determined by the ordered subset of their atoms and dual atoms. Duffus and Rival essentially proved the following result.
	
	\begin{theorem}[{\cite[Theorem 3.3]{Duff}}]{\label{duff}}
		Let $L_1$ and $L_2$ be two lattices of finite length such that $Cov(L_1) \cong Cov(L_2)$. Then  $L_1$ is atomistic and dual atomistic if and only if $L_2$ is atomistic and dual atomistic. Moreover, if this condition is satisfied then there are lattices $ A, B $ such that $L_1 \cong A \times B $ and $L_2 \cong A^d \times B$. 
	\end{theorem}

	Later on, the same result was extended by Stern \cite{stern} to the class of balanced lattices and gave a common generalization and essentially proved the following result.
	
	\begin{theorem}[{\cite[Theorem 7]{stern}}]
		Let $L_1$ and $L_2$  be  lattices of finite length such that $Cov(L_1) \cong Cov(L_2)$. Then $L_1$ is balanced if and only if $L_2$ is balanced. Moreover, if this condition is satisfied then there are sublattices $A$ and $B$ of $L_1$ such that $L_1 \cong A \times B$ and $L_2 \cong A^d \times B$.
	\end{theorem}
	
	This motivates us to think about the same problems for the comparability graphs, i.e., to find a necessary and sufficient condition on lattices $L_1$ and $L_2$ such that graph isomorphism gives lattice isomorphism. Precisely, we have:
	
	\textbf{Isomorphism Problem: Find a class $ \mathscr{L} $ of lattices such that $Com(L_1) \cong Com(L_2)$ if and only if $L_1 \cong L_2$, for $L_1, L_2 \in \mathscr{L} $.}
	
	Note that the covering graph isomorphism of posets does not imply the comparability graph isomorphism. It is clear from the following figure in which $P_1$ and $P_2$ have the same covering graphs but different comparability graphs. 
	
	
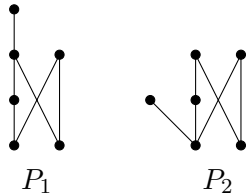
\begin{center}
		
		\begin{tikzpicture}[scale =0.6]
			
			\draw [fill=black] (0,3) circle (.1); 
			\draw [fill=black] (0,0) circle (.1); 
			\draw [fill=black] (0,2) circle (.1);  
			\draw [fill=black] (0,1) circle (.1); 
			\draw [fill=black] (1,2) circle (.1); 
			\draw [fill=black] (1,0) circle (.1);
			
			\draw (0,0)--(0,1)--(0,2)--(0,3);
			
			\draw
			(0,0)--(1,2)--(1,0)--(0,2);
			
			\draw node [below] at (0.5,-0.3) {$P_1$};
			
			\begin{scope}[shift={(4,0)}]
				
				\draw [fill=black] (-1,1) circle (.1); 
				\draw [fill=black] (0,0) circle (.1); 
				\draw [fill=black] (0,2) circle (.1);  
				\draw [fill=black] (0,1) circle (.1); 
				\draw [fill=black] (1,2) circle (.1); 
				\draw [fill=black] (1,0) circle (.1); 
				
				\draw (0,0)--(-1,1);
				
				\draw (0,0)--(0,1)--(0,2);
				
				\draw
				(0,0)--(1,2)--(1,0)--(0,2);
				
				\draw node [below] at (0.5,-0.3) { $P_2$};
				
			\end{scope}
			
			
		\end{tikzpicture}
		
	\end{center}
	
	A \textit{comparability graph} is a simple, unoriented graph in which two vertices are adjacent if and only if they are comparable with regard to some partial order on the vertices. It is also known as a transitively orientable graph. For example, it is easy to see that every complete graph and a bipartite graph is a comparability graph. 
	
    In this paper, we addressed the following crucial problems about the comparability graph of lattices.

\begin{enumerate}
	\item What are the properties of lattices that are preserved under the graph isomorphism? 
	
	\item To determine the classes of lattices in which the graph isomorphism gives the lattice isomorphism.
\end{enumerate}

	In \cite{rjvj}, Jejurkar and Joshi  unified many results proved for
	connectedness, girth and diameter of the inclusion graphs of algebraic structures given in \cite{SA}, \cite{ad}, \cite{Ad}, \cite{devi}, \cite{JG} using the comparability graphs of the lattice of substructures of the corresponding algebraic structure. One can think about the Isomorphism Problem in these inclusion graphs. There are several approaches to tackle this problem. One approach is to think algebraically, while another is to approach the problem from a lattice theoretic perspective.
	
	With this brief introduction, we begin the paper with some basic definitions and terms.
	
\section{Preliminaries}
    Let $(L;\leq)$ be a lattice. The least element and the greatest element (if they exist) will be denoted by $0_L$ and $1_L$, respectively. A subset $L'$ of $L$ is said to be a \textit{sublattice} of $L$ if it satisfies the property that $a, b \in L'$ implies that $a \vee b, a \wedge b \in L'$ with the $\vee $ and the $ \wedge $ of $L'$ are restrictions to $L'$ of the $\vee $ and the $ \wedge $ of $L$. A lattice $L$ is said to be \textit{bounded} if it has both $0_L$ and $1_L$. Two elements $a, b$  are incomparable, then we denote it by $a \| b$. If any two elements in $L$ are comparable, then $L$ is said to be a \textit{chain}. A chain with $n$ elements is denoted by $C_n$ and the length of $C_n$ is $n-1$. If $L$ has $0_L$ and every chain in $L$ is finite, then the \textit{height} $h(a)$ of an element $a \in L$ is the maximum length of a maximal chain from $0_L$ to $a$. The \textit{length} of $L$, denoted by $\ell(L)$, is defined as the  maximum length of a maximal chain in $L$. If $L$ has $1_L$, then $\ell(L) = h(1_L)$. 
	
	For $x, y \in L$, we write $y \Yleft x$ ($y$ is \textit{covered } by $x$ or $x$ \textit{covers} $y$) if $y < x$ and $y < z \leq x$ implies that $x = z$. An element $ a $ of a lattice $ L $ with $0_L$ is an \textit{atom}, if $0_L \Yleft a$ and it is a \textit{dual atom}, if $a \Yleft 1_L$. \textit{The set of atoms is denoted by $At(L)$ and the set of dual atoms is denoted by $DualAt(L)$.} A lattice $ L $ with $0_L$ is called \textit{atomic}, if for every non-zero element $b$ there exists an atom $a \in L$ such that $ a \leq b$.   An atomic lattice $L$ is called \textit{atomistic} if every nonzero element is a join of atoms. It is called \textit{dual atomistic} if every nonunit element is a meet of dual atoms.
	
	A bijective map $\phi$ from a lattice $L_1$ to a lattice $L_2$ is said to be an isomorphism, if it preserves meet and join operation, i.e., $\phi(a \wedge b) = \phi(a) \wedge \phi(b)$ and $\phi(a \vee b) = \phi(a) \vee \phi(b)$. Equivalently, the bijective map $\phi$ is an \textit{isomorphism}, if it is \textit{bi-order preserving map}, i.e.,  $a \leq b$ if and only if  $\phi(a) \leq \phi(b)$. Two lattices $L_1$ and $L_2$ are isomorphic, then we denote it by $L_1 \cong L_2$.
	
	Let  $(L, \leq)$ be a lattice with the partial order $\leq$. Then the \textit{dual lattice} $(L, \geq)$ of $L$  is the lattice, where the partial order $\geq $ is defined as $x \geq y$ if and only if $x \leq y$ in $L$ for $x, y \in L$. $ L^d$ denotes the dual lattice of $L$. A lattice is said to be \textit{self dual} if $L \cong L^d$.  
	
	A lattice is called \textit{upper semimodular} if and only if it satisfies the \textit{upper covering condition}, that is,
	$a \Yleft b$ implies that $a \vee c \Yleft b \vee c$ or $a \vee c = b \vee c$.	The upper semimodular lattices are also known as semimodular lattices. The dual notion of upper semimodular is \textit{lower semimodular}. A lattice $L$ is called \textit{modular}, if for $a, b, c \in L$, $c \vee (a \wedge b) = (c \vee a) \wedge b$ for all $c \leq b$. Clearly, a modular lattice is upper semimodular as well as lower semimodular. A lattice $L$ is said to be \textit{distributive} if it satisfies $a \wedge (b \vee c) = (a \wedge b) \vee (a \wedge c)$ or $a \vee (b \wedge c) = (a \vee b) \wedge (a \vee c)$ for all $a, b, c \in L$. A lattice $L$ with least element $0$ is said to be \textit{$0$-distributive} if for any $a,b,c \in L$, $a \wedge b = 0 $ and $a \wedge c = 0 $ implies $a \wedge (b \vee c) = 0 $. The dual notion of 0-distributive lattice is \textit{$1$-distributive} lattice. In a bounded lattice $ L $, an element $ a $ is a \textit{complement} of an element $ b $, if $a \vee b = 1_L$ and $a \wedge b = 0_L$ and the lattice $L$ is said to be \textit{complemented} if every element in $L$ has a complement. A \textit{Boolean lattice} is a lattice that is complemented and distributive. For any undefined concepts related to lattice theory can be found in \cite{Birkhoff}, \cite{Gratzer}, and \cite{stern}.
	
	Let $G$ be a graph with the vertex set $V(G)$ and the edge set $E(G)$. A graph $G$ with $V(G) \neq \phi$ is an \textit{edgeless graph} if $E(G)$ is empty. A \textit{simple graph} is a graph having no loops and multiple edges. A graph $ H $ is said to be a \textit{subgraph} of a graph  $G$ if $V(H) \subseteq V(G) $ and $E(H) \subseteq E(G) $.  If all the vertices of $G$ are pairwise adjacent, then $G$ is said to be \textit{complete}. Two graphs $G_1$ and $G_2$ are said to be \textit{isomorphic} if there exists a bijective map, $\phi $, from $V(G_1)$ to $V(G_2)$ such that $(u, v) \in E(G_1)$ if and only if $(\phi (u), \phi (v)) \in E(G_2)$, for any $u, v \in V(G_1)$. A finite graph $G$ is a \textit{path} if its vertices can be ordered so that two vertices are adjacent if and only if they are consecutive in the list and its \textit{length} is the number of edges in it. A path $G$ is a \textit{cycle} if its initial and end vertices coincide. A graph $G$ is \textit{connected}, if each pair of vertices in $G$ belongs to a path, otherwise, $G$ is \textit{disconnected}. Any undefined concepts related to graph theory can be found in \cite{west}.

	\textbf{Throughout this paper, $L$ is a  bounded lattice, and all graphs under consideration are connected graphs.}
	
	Let's start with a formal definition of the comparability graph of a lattice.
	
	\begin{definition}		
		Let $L$ be a bounded lattice. The \textit{ comparability graph} of $L$ is an undirected, simple graph denoted by $Com(L)$, where the vertex set is $L\setminus\{0_L,1_L\}$ and two vertices $a$ and $b$ are adjacent if and only if $a$ and $b$ are comparable, i.e., $a < b$ or $b < a$.
	\end{definition}
	
	\begin{example}
		The following figure shows the lattices $N_5$ and $M_3$ and their comparability graphs.
	\end{example}
	
	
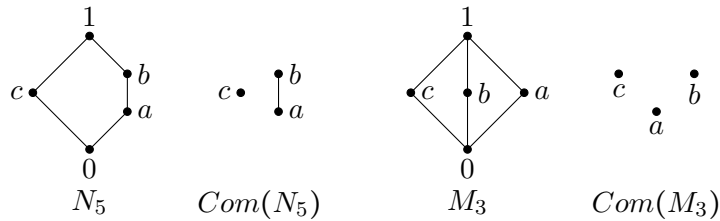
\begin{center}
		
		\begin{tikzpicture}[scale =0.5]
			
			\draw [fill=black] (0,0) circle (.1); 
			\draw [fill=black] (-1.5,1.5) circle (.1); 
			\draw [fill=black] (1,1) circle (.1);
			\draw [fill=black] (1,2) circle (.1);
			\draw [fill=black] (0,3) circle (.1);
			
			\draw (0,0)--(-1.5,1.5)--(0,3);
			\draw (0,0)--(1,1)--(1,2)--(0,3);
			
			\draw node [below] at (0,0) {$0$};
			\draw node [left] at (-1.5,1.5) {$c$};
			\draw node [right] at (1,1) {$a$};
			\draw node [right] at (1,2) {$b$};
			\draw node [above] at (0,3) {$1$};
			\draw node [below] at (0,-0.8) {$N_5$};
			
			\begin{scope}[shift={(4,0)}]
				
				\draw [fill=black] (0,1.5) circle (.1); 
				\draw [fill=black] (1,1) circle (.1);
				\draw [fill=black] (1,2) circle (.1);
				
				\draw (1,1)--(1,2);
				
				\draw node [left] at (0,1.5) {$c$};
				\draw node [right] at (1,1) {$a$};
				\draw node [right] at (1,2) {$b$};

				\draw node [below] at (0.5,-0.8) {$Com(N_5)$};
				
			\end{scope}
			\begin{scope}[shift={(10,0)}]
				
				\draw [fill=black] (0,0) circle (.1); 
				\draw [fill=black] (1.5,1.5) circle (.1);
				\draw [fill=black] (0,1.5) circle (.1);
				\draw [fill=black] (-1.5,1.5) circle (.1);
				\draw [fill=black] (0,3) circle (.1);

				\draw (0,0)--(1.5,1.5)--(0,3)--(-1.5,1.5)--(0,0);
				\draw (0,0)--(0,1.5)--(0,3);
				
				\draw node [right] at (-1.5,1.5) {$c$};
				\draw node [right] at (0,1.5) {$b$};
				\draw node [right] at (1.5,1.5) {$a$};
				\draw node [below] at (0,0) {$0$};
				\draw node [above] at (0,3) {$1$};

				\draw node [below] at (0,-0.8) {$M_3$};
				
			\end{scope}
			\begin{scope}[shift={(15,0)}]
				
				\draw [fill=black] (0,1) circle (.1); 
				\draw [fill=black] (1,2) circle (.1); 
				\draw [fill=black] (-1,2) circle (.1); 
				
				\draw node [below] at (0,1) {$a$};
				\draw node [below] at (1,2) {$b$};
				\draw node [below] at (-1,2) {$c$};
				\draw node [below] at (0,-0.8) {$Com(M_3)$};
				
			\end{scope}
			
		\end{tikzpicture}
	\end{center}

\begin{remark}
	A lattice and its dual have isomorphic comparability graphs.	
\end{remark}
	
\section{properties that are preserved under graph isomorphism}

	In order to tackle the Isomorphism Problem for comparability graphs, one should be aware of the lattice theoretical properties that are preserved under the graph isomorphism. It means that if we have two lattices $L_1$ and $L_2$ with isomorphic comparability graphs, then $L_1$ possesses a property if and only if $L_2$ possesses the same property. This helps to get a class of lattices in which graph isomorphism gives lattice isomorphism and vice versa.
	
	Let's see some lattice theoretic properties that are preserved under graph isomorphism.

\begin{theorem} \label{substr}
    Let $L_1$ and $L_2$ be two lattices such that $Com(L_1) \cong Com(L_2)$. Then 
    \begin{enumerate}
      	      	
      	\item $L_1$ contains $N_5$ as a sublattice if and only if $L_2$ also contains it.
      	
      	\item $L_1$ contains $M_3$ as a sublattice if and only if $L_2$ also contains it.
      \end{enumerate}    
\end{theorem}	

\begin{proof}
	Let $L_1$ and $L_2$ be two lattices such that $Com(L_1) \cong Com(L_2)$ and $\phi$ be a graph isomorphism from $Com(L_1)$ to $Com(L_2)$.

   \textbf{Proof of (1):} Suppose that $L_1$ contains  a sublattice $A =\{a_1,a_2,a_3,a_1\wedge a_3, a_2\vee a_3\} $   isomorphic to $N_5$, where $a_1 < a_2$, $a_3\Vert a_1$, $a_3\Vert a_2$, $a_1\wedge a_3 = a_2\wedge a_3$ and $a_1\vee a_3 = a_2\vee a_3$. Since $a_1 < a_2$, i.e., $a_1 \sim a_2$ in $Com(L_1)$, we have $\phi(a_1) \sim \phi(a_2)$ in $Com(L_2)$, i.e., either $\phi(a_1) < \phi(a_2)$ or $ \phi(a_2) < \phi(a_1) $.
   
   Without loss of generality, we assume that $\phi(a_1) < \phi(a_2)$. As $a_3\Vert a_1$ and  $a_3\Vert a_2$, we have $a_3 \nsim a_1$ and $a_3 \nsim a_2$ in $Com(L_1)$. This gives $\phi(a_3) \nsim \phi(a_1)$ and $\phi(a_3) \nsim \phi(a_2)$ in $Com(L_2)$, i.e., $\phi(a_3) \Vert \phi(a_1)$  and $\phi(a_3) \Vert \phi(a_2)$. Hence  $\phi(a_3) \wedge \phi(a_1) < \phi(a_1) < \phi(a_2)$ and $\phi(a_3) \wedge \phi(a_1) < \phi(a_3)$.  This gives $\phi(a_3) \wedge \phi(a_1) \le \phi(a_3) \wedge \phi(a_2)$. 
   
   Now, we claim that  $\phi(a_3) \wedge \phi(a_1) = \phi(a_3) \wedge \phi(a_2)$. 
   
   Suppose on the contrary that $\phi(a_3) \wedge \phi(a_1) \neq \phi(a_3) \wedge \phi(a_2)$. Hence $\phi(a_3) \wedge \phi(a_1) < \phi(a_3) \wedge \phi(a_2)$. This proves that $\phi(a_3) \wedge \phi(a_2) \in V(Com(L_2))$.
   
   Let $c$ be an element of $L_1$ such that $\phi(c) = \phi(a_3) \wedge \phi(a_2)$. As $\phi(a_3) \sim \phi(a_3) \wedge \phi(a_2) = \phi(c) $  and $\phi(a_2) \sim \phi(a_3) \wedge \phi(a_2) = \phi(c) $, we have $a_3 \sim c$ and $a_2 \sim c$. Since $a_3 \| a_2$, we can not have $a_3 < c < a_2$ and  $ a_2 < c < a_3$. Therefore we consider the following cases.
   
   \textbf{Case(1):} If $ c < a_3  $ and $ c < a_2$, then $ c \leq a_3 \wedge a_2 $. However, $a_3 \wedge a_2 = a_3 \wedge a_1$ gives $c \leq a_3 \wedge  a_1 < a_1$, i.e., $a_1 \sim c$. This gives $ \phi(a_1) \sim \phi(c) $, i.e.,  $ \phi(a_1) \sim \phi(a_3) \wedge \phi(a_2)$. Hence we have either $ \phi(a_1) < \phi(a_3) \wedge \phi(a_2)$ or $ \phi(a_3) \wedge \phi(a_2) < \phi(a_1)$. If $ \phi(a_1) < \phi(a_3) \wedge \phi(a_2)$, then $\phi(a_3) \wedge \phi(a_2) < \phi(a_3)$ gives $\phi(a_1) < \phi(a_3)$, i.e., $\phi(a_1) \sim \phi(a_3)$, a contradiction. Thus, we must have $ \phi(a_3) \wedge \phi(a_2) < \phi(a_1)$. Since $\phi(a_3) \nsim \phi(a_2)$, we have $ \phi(a_3) \wedge \phi(a_2) < \phi(a_3)$, we have $ \phi(a_3) \wedge \phi(a_2) \leq \phi(a_3) \wedge \phi(a_1)$. Hence  $ \phi(a_3) \wedge \phi(a_1) = \phi(a_3) \wedge \phi(a_2)$.
   
   \textbf{Case(2):} If $a_3 < c$ and $a_2 < c$, then $a_3 \vee a_2 \leq c $. However, $a_3 \vee a_2 = a_3 \vee a_1$ gives $a_1 < a_3 \vee a_1 \leq c$, i.e., $a_1 \sim c$. This implies that $ \phi(a_1) \sim \phi(c) $, i.e.,  $ \phi(a_1) \sim \phi(a_3) \wedge \phi(a_2)$. Therefore we have either $ \phi(a_1) < \phi(a_3) \wedge \phi(a_2)$ or $ \phi(a_3) \wedge \phi(a_2) < \phi(a_1)$.  If $ \phi(a_1) < \phi(a_3) \wedge \phi(a_2)$, then $\phi(a_3) \wedge \phi(a_2) < \phi(a_3)$ gives $\phi(a_1) < \phi(a_3)$, i.e., $\phi(a_1) \sim \phi(a_3)$, a contradiction. Thus, we must have $\phi(a_3) \wedge \phi(a_2) <  \phi(a_1)$. Since $\phi(a_3) \wedge \phi(a_2) <  \phi(a_3)$. Hence $ \phi(a_3) \wedge \phi(a_1) \geq \phi(a_3) \wedge \phi(a_2)$. Hence $ \phi(a_3) \wedge \phi(a_1) = \phi(a_3) \wedge \phi(a_2)$.
   
   Thus, from both the cases,  $ \phi(a_3) \wedge \phi(a_1) = \phi(a_3) \wedge \phi(a_2)$.
   
   Again, we have $ \phi(a_1) < \phi(a_2) < \phi(a_3) \vee \phi(a_2)$ and $\phi(a_3) < \phi(a_3) \vee \phi(a_2)$ gives $\phi(a_3) \vee \phi(a_1) \le \phi(a_3) \vee \phi(a_2)$. Arguing on the similar way as above, we get  $\phi(a_3) \vee \phi(a_1) = \phi(a_3) \vee \phi(a_2)$.
   
   Thus, the set $A' = \{\phi(a_1), \phi(a_2), \phi(a_3), \phi(a_1) \wedge \phi(a_3), \phi(a_2) \vee \phi(a_3)\}$ forms a sublattice of $L_2$ isomorphic to $N_5$.

   \textbf{Proof of (2):} Suppose that  $L_1$ contains $M_3$ as a sublattice. Let $A = \{a, b, c, a \wedge b = b \wedge c = a \wedge c = a \wedge b \wedge c, a \vee b = b \vee c = a \vee c = a \vee b \vee c \}$ be a sublattice of $L_1$ isomorphic to $M_3$, where $a \| b, a\| c, b\|c$.
   
   Since $a \| b, a\| c$ and $b\|c$, we have $ \phi(a) \| \phi(b),  \phi(a)\| \phi(c)$ and $\phi(b)\|\phi(c)$. Clearly, $\phi(a) \wedge \phi(b) \wedge \phi(c) \leq \phi(a) \wedge \phi(b)$. We claim that $\phi(a) \wedge \phi(b) = \phi(a) \wedge \phi(b) \wedge \phi(c)$. Suppose on the contrary that $\phi(a) \wedge \phi(b) \wedge \phi(c) < \phi(a) \wedge \phi(b)$. Hence $\phi(a) \wedge \phi(b) \neq 0$. Thus $\phi(a) \wedge \phi(b) \in V(Com(L_2))$.
   
   Let $\phi(a) \wedge \phi(b) = \phi(d)$, for some $d \in L_2$. Since $\phi(a) \wedge \phi(b) \sim \phi(a),  \phi(b)$, we have $\phi(d) \sim \phi(a), \phi(b)$ and hence $d \sim a$ and  $d \sim b $. Clearly, $a \leq d \leq b$ and $b \leq d \leq a$ will not be possible. Therefore we consider the following cases.
   
   \textbf{Case(1):} If $ d\leq a$ and $ d\leq b$, then $ d \leq a \wedge b = b \wedge c < c $ gives $d \leq c$, i.e., $d \sim c$. Hence $\phi(d) \sim \phi(c)$, i.e., $\phi(a) \wedge \phi(b) \sim \phi(c)$. If $\phi(c) \leq \phi(a) \wedge \phi(b)$, then $\phi(c) \leq \phi(a) $, a contradiction to $a \| c$. Hence we have $\phi(a) \wedge \phi(b) \leq  \phi(c)$ this gives $\phi(a) \wedge \phi(b)  = \phi(a) \wedge \phi(b) \wedge \phi(c)$. 
   
   \textbf{Case(2):} If $a \leq d$ and $b \leq d$, then $a \vee b \leq d$. This together with $c < a\vee c = a\vee b \leq d$ gives $\phi(c) \sim \phi(d)$, i.e., $ \phi(c) \sim  \phi(a) \wedge \phi(b)$. If $\phi(c) \leq \phi(a) \wedge \phi(b)$, then $\phi(c) \leq \phi(a)$, a contradiction to $a \| c$. Hence we have $\phi(a) \wedge \phi(b) \leq  \phi(c)$. This gives $\phi(a) \wedge \phi(b)  = \phi(a) \wedge \phi(b) \wedge \phi(c)$. 
   
   Thus, from both cases, we have $\phi(a) \wedge \phi(b)  = \phi(a) \wedge \phi(b) \wedge \phi(c)$.
   
   Again, we have $\phi(a) \vee \phi(b) \leq \phi(a) \vee \phi(b) \vee \phi(c)$. Arguing on the similar way as above, we have $\phi(a) \vee \phi(b) = \phi(a) \vee \phi(b) \vee \phi(c)$.
   
   On the similar lines, we can show that $\phi(a) \wedge \phi(c)= \phi(b) \wedge \phi(c)  = \phi(a) \wedge \phi(b) \wedge \phi(c)$ and $\phi(a) \vee \phi(c)= \phi(b) \vee \phi(c) = \phi(a) \vee \phi(b) \vee \phi(c)$ and hence the set $A' = \{\phi(a), \phi(b), \phi(c), \phi(a) \wedge \phi(b) = \phi(a) \wedge \phi(c) = \phi(b) \wedge \phi(c), \phi(a) \vee \phi(b)= \phi(a) \vee \phi(c) = \phi(b) \vee \phi(c)\}$ forms a sublattice of $L_2$ isomorphic to $M_3$. 
\end{proof}	
	
\begin{theorem}[{\cite[p. 58]{Gratzer}}] \label{pentagon}
	A lattice $ L $ is modular if and only if $ L $ does not contain a pentagon ($N_5$) as a sublattice.
\end{theorem}
	
\begin{corollary}\label{modular}
	Let $L_1$ and $L_2$ be two lattices such that $Com(L_1) \cong Com(L_2)$. Then $L_1$ is modular if and only if $L_2$ is modular.
\end{corollary}
	
\begin{proof} 
	This result is  follows from Theorem \ref{substr} and \ref{pentagon}.		
\end{proof}
	
\begin{theorem}[{\cite[p. 58]{Gratzer}}] \label{diamond}
	A lattice $ L $ is distributive if and only if $ L $ does not contain a pentagon ($N_5$) or a diamond ($M_3$).
\end{theorem}
	
\begin{corollary}\label{dist}
	Let $L_1$ and $L_2$ be two lattices such that $Com(L_1) \cong Com(L_2)$. Then $L_1$ is distributive if and only if $L_2$ is distributive.
\end{corollary}

\begin{proof}
	This result is  follows from Theorem \ref{substr} and \ref{diamond}.	
\end{proof}

\begin{definition}
	For an integer $n \geq 3$ a \textit{crown} is a poset $\{x_1, y_1, x_2, y_2,\dots, x_n, y_n \}$ in which $x_i \leq y_i$ for $i=1, 2, \dots, n$, $x_{i+1} \leq y_i$ for $ i=1, 2, \dots, n-1$ and $x_1 \leq y_n$ are the only comparability relations.
\end{definition}

	
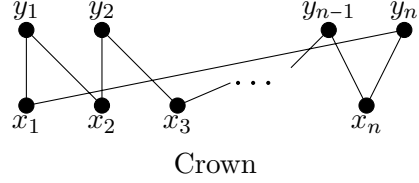
\begin{center}
	\begin{tikzpicture}[scale =1]
		 
		 \draw [fill=black] (1.5,0) circle (.1);  
		 \draw [fill=black] (2,1) circle (.1);
		 \draw [fill=black] (-1,0) circle (.1); 
		 \draw [fill=black] (-2,0) circle (.1); 
		 \draw [fill=black] (-3,0) circle (.1); 
		 \draw [fill=black] (-3,1) circle (.1); 
		 \draw [fill=black] (-2,1) circle (.1); 
		 \draw [fill=black] (1,1) circle (.1);

		 \draw [fill=black] (-0.2,0.3) circle (.02);
		 \draw [fill=black] (0,0.3) circle (.02);
		 \draw [fill=black] (0.2,0.3) circle (.02); 
		 
		 \draw  (0.5,0.5)--(1,1)--(1.5,0)--(2,1)--(-3,0)--(-3,1)--(-2,0)--(-2,1)--(-1,0)--(-0.3,0.3);
		
		\draw node [below] at (-3,0) { $x_1$};
		\draw node [below] at (-2,0) { $x_2$};
		\draw node [below] at (-1,0) { $x_3$};
		\draw node [below] at (1.5,0) { $x_n$};
		
		\draw node [above] at (-3,1) { $y_1$};
		\draw node [above] at (-2,1) { $y_2$};
		\draw node [above] at (1,1) { $y_{n-1}$};
		\draw node [above] at (2,1) { $y_n$};
		
		\draw node [below] at (-0.5,-0.5) {Crown};
	
	\end{tikzpicture}
	
\end{center}

\begin{theorem}\label{crown} 
	Let $L$ be a lattice. Then there is a one-to-one correspondence between crowns in $L$ and induced cycles of length at least 6 in $Com(L)$. 
\end{theorem}

\begin{proof}
	Let $L$ be a lattice.
	
	\textbf{Part(1):} Crowns gives induced cycles in $Com(L)$.
	
	Let $\{x_1, y_1, x_2, y_2,\dots, x_n, y_n \}$,  $n \geq 3$ be a crown in $L$. Clearly, $x_i, y_i$ are nonzero, nonunit in $L$. We have $x_i \leq y_i$   $(i=1, 2, \dots, n)$, $x_{i+1} \leq y_i$   $(i=1, 2, \dots, n-1)$ and $x_1 \leq y_n$ are the only comparability relations. let $H$ be a subgraph induced by $\{x_1, y_1, x_2, y_2,\dots, x_n, y_n \}$. Clearly, $x_i \sim y_i$  $  (i=1, 2, \dots, n)$, $x_{i+1} \sim y_i$  $  (i=1, 2, \dots, n-1)$ and $x_1 \sim y_n$ are the only adjacency in $Com(L)$. Therefore we have an induced cycle $x_1 \sim y_1 \sim x_2 \sim y_2 \dots \sim x_{n} \sim y_n \sim x_1$ in $Com(L)$.
	
    \textbf{Part(2):} Induced cycles in $Com(L)$ gives crown in $L$.
    
    By Lemma 6.4 and Lemma 6.5 of \cite{rjvj}, we have $Com(L)$ does not contains induced  cycles $C_k$ ($k \geq 5$) of odd length and induced cycles of length 4.
    
    Let $C = : a_1 \sim b_1 \sim a_2 \sim b_2 \sim \dots \sim b_n \sim a_1$ be an induced cycle in $Com(L)$. Clearly, length of $C$ is even and greater than or equal to 6. Since $C$ is an induced cycle, we have $a_i \sim b_i$ $ (1 \leq i \leq n)$,  $a_i \sim b_{i-1}$ $ (2 \leq i \leq n)$ and $b_n \sim a_1$ are the only adjacency in $Com(L)$. As $a_1 \sim b_1$, we have $a_1 < b_1$ or $b_1 < a_1$. Without loss of generality assume that $a_1 < b_1$. Also $b_1 \sim a_2$ gives either $b_1 < a_2$ or $a_2 < b_1  $. If $b_1 < a_2$, then $a_1 < b_1< a_2$ gives $a_1 \sim a_2$, a contradiction. So that we have $a_2 < b_1  $. Again $a_2 \sim b_2$ gives either $a_2 < b_2$ or $b_2 < a_2$. If $b_2 < a_2$ then $b_2< a_2< b_1$ gives $b_1 \sim b_2$, a contraction and hence $a_2 < b_2$. Continuing in this way, we get $a_3 < b_3, \dots, a_i < b_i, a_{i+1} < b_i, \dots a_n < b_n$. Lastly, $b_n \sim a_1$ gives $b_n < a_1$ or $a_1 < b_n$. If $b_n < a_1$ then $b_n < a_1 < b_1$ gives $b_n \sim b_1$, a contradiction. So that we have $a_1 < b_n$. Hence we can conclude that $a_i < b_i$ $1 \leq i \leq n$, $a_{i+1} < b_i$ $2 \leq i \leq n$ and $a_1 < b_n$ are the only comparability relations in $L$. Therefore we can say that $a_1, b_1, \dots , a_n, b_n $ forms a crown in $L$.
\end{proof}

\begin{theorem}
	 Let $L_1$ and $L_2$ be two lattices such that $Com(L_1) \cong Com(L_2)$. Then $L_1$ contains a crown if and only if $L_2$ also contains it.
\end{theorem}

\begin{proof}
	Let $\phi$ be a graph isomorphism from $Com(L_1)$ to $Com(L_2)$. Suppose that $L_1$ contains a crown  $\{x_1, y_1, x_2, y_2,\dots, x_n, y_n \}$,  $n \geq 3$ with  $x_i \leq y_i$   $(i=1, 2, \dots, n)$, $x_{i+1} \leq y_i$   $(i=1, 2, \dots, n-1)$ and $x_1 \leq y_n$. By Theorem \ref{crown}, it corresponds to an induced cycle in $Com(L_1)$. Since $Com(L_1) \cong Com(L_2)$, we have $\{\phi(x_1), \phi(y_1), \dots, \phi(x_n), \phi(y_n)\}$ is an induced cycle in $Com(L_2)$. Again by Theorem \ref{crown}, it will form a crown in $L_2$.
\end{proof}

\begin{definition}[{\cite[p. 51]{Gratzer}}]
	A finite lattice $L$ of $n$ elements is dismantlable if and only if there is a chain $ L_1 \subset L_2 \subset \dots \subset L_n = L $ of sublattices satisfying $|L_i| = i$.
\end{definition}
    
\begin{theorem}[{\cite[Theorem 3.2]{rival}}]\label{dismantlable}
	A lattice which contains no infinite chains and no infinite fences is dismantlable if and only if it contains no crowns.
\end{theorem}

\begin{corollary}
	Let $L_1$ and $L_2$ be two lattices which contains no infinite chains and no infinite fences  such that $Com(L_1) \cong Com(L_2)$. Then $L_1$ is dismantlable if and only if $L_2$ is dismantlable.
\end{corollary}

 A graph G is chordal if it contains no induced cycles of length more than 3. A graph $ G $ is distance-hereditary if for any two vertices $ u $ and $ v $ belonging to a connected induced subgraph $ H $ of $ G $, some shortest path connecting $ u $ and $ v $ in G lies in $ H $. Equivalently, distance-hereditary graphs are the graphs in which every induced path is a shortest path. A graph $ G $ is Ptolemaic if it is distance-hereditary and chordal.

\begin{corollary}
	Let $L$ be a finite lattice. Then $Com(L)$ is chordal if and only if $L$ is dismantlable.
\end{corollary}

\begin{proof}
	By Lemma 6.4 and Lemma 6.5 of \cite{rjvj}, we have $Com(L)$ does not contains induced  cycles $C_k$ of odd length, $k \geq 5$ and induced cycles of length 4. Therefore by Theorem \ref{crown} and Theorem \ref{dismantlable}, result follows.
\end{proof}

\begin{corollary}\label{dist-here}
	Let $L$ be a finite lattice. If $Com(L)$ is distance-hereditary then $L$ is dismantlable lattice.
\end{corollary}
    
\begin{proof}
	 Let $L$ be a finite lattice. Suppose that $Com(L)$ is distance-hereditary. We have to show that $L$ is a dismantlable lattice. On the contrary suppose that $L$ is not dismantlable. By Theorem \ref{dismantlable}, we can say $L$ contains a crown. Let $\{x_1, y_1, x_2, y_2,\dots, x_n, y_n \}$,  $n \geq 3$ be a crown in $L$ with  $x_i \leq y_i$   $(i=1, 2, \dots, n)$, $x_{i+1} \leq y_i$   $(i=1, 2, \dots, n-1)$ and $x_1 \leq y_n$. By Theorem \ref{crown}, it forms an induced cycle in $Com(L)$. Let $C = : x_1 \sim y_1 \sim x_2 \sim y_2 \sim \dots \sim y_n \sim x_1$ be the induced cycle in $Com(L)$. Consider a path $P_1 =: x_1 \sim y_1 \sim x_2 \sim y_2 \sim \dots \sim y_{n-1} \sim x_n$. Since $C$ is an induced cycle, $P_1$ is an induced path in $Com(L)$. Clearly, length of $P_1$ is at least 4. Now consider another path $P_2 =: x_1 \sim y_n \sim x_n$. Clearly, length of $P_2$ is 2. Thus $P_1$ is not a shortest path. Hence $Com(L)$ is not distance-hereditary which is a contradiction to our assumption. Thus $L$ must be a dismantlable lattice.
\end{proof}

By using Corollary \ref{dist-here}, following corollary is straightforward.

\begin{corollary}
	Let $L$ be a finite lattice. If $Com(L)$ is distance-hereditary then it is chordal and hence ptolemaic.
\end{corollary}
	
   Up till now, we showed that some lattice theoretic properties which have forbidden substructure characterization are preserved. Interestingly, in the next few results, we will show that the other properties which do not have forbidden substructure characterization, such as complementation, and atomisticity with dual atomisticity etc., are also preserved. 
	
	\begin{theorem}\label{comp}
		Let $L_1$ and $L_2$ be two lattices such that $Com(L_1) \cong Com(L_2)$. Then $L_1$ is complemented if and only if $L_2$ is complemented.
	\end{theorem}
	
	\begin{proof}
		Let $ \phi $ be a graph isomorphism from $Com(L_1)$ to $Com(L_2)$. Assume that $L_1$ is a complemented lattice. We have to show that $L_2$ is also a complemented lattice. Let $a $ be an element in $L_2$. On the contrary, assume that $a$ does not have a complement in $L_2$. Since $L_1$ is complemented lattice, $\phi^{-1}(a)$ have complement in $L_1$. Let $\phi^{-1}(a')$ be a complement of $\phi^{-1}(a)$, for some $a'$ in $L_2$. Hence we have $\phi^{-1}(a) \wedge \phi^{-1}(a') = 0_{L_1}$ and $\phi^{-1}(a) \vee \phi^{-1}(a') = 1_{L_1}$. This gives $\phi^{-1}(a) \| \phi^{-1}(a')$ in $L_1$, i.e., $a \| a'$ in $L_2$. Since $a$ does not have complement in $L_2$, we have either $a \wedge a' \neq 0_{L_2}$ or $a \vee a' \neq 1_{L_2}$. First assume that $a \wedge a' \neq 0_{L_2}$. Let $a \wedge a' = b$,  for some $ b \in L_2 \setminus \{0_{L_2},1_{L_2}\}$. Since $a > a \wedge a' = b$ and $a' > a \wedge a' = b$, we have $a \sim b $ and $a' \sim b $ in $Com(L_2)$. Hence $\phi^{-1}(a) \sim \phi^{-1}(b) $ and $\phi^{-1}(a') \sim \phi^{-1}(b) $ in $Com(L_1)$.
		
		If $\phi^{-1}(a) > \phi^{-1}(b) > \phi^{-1}(a') $  and $\phi^{-1}(a) < \phi^{-1}(b) < \phi^{-1}(a')$, then $\phi^{-1}(a)$ and $\phi^{-1}(a')$ will be comparable, which is not possible. 	 
		
		If $\phi^{-1}(a) > \phi^{-1}(b)$ and $\phi^{-1}(a') > \phi^{-1}(b)$, then $\phi^{-1}(a) \wedge \phi^{-1}(a') \geq \phi^{-1}(b)$. However, $\phi^{-1}(a) \wedge \phi^{-1}(a') = 0_{L_1}$. This gives  $\phi^{-1}(b) = 0_{L_1}$, a contradiction. 
		
		If $\phi^{-1}(a) < \phi^{-1}(b)$ and $\phi^{-1}(a') < \phi^{-1}(b)$, then $\phi^{-1}(a) \vee \phi^{-1}(a') \leq \phi^{-1}(b)$. However, $\phi^{-1}(a) \vee \phi^{-1}(a') = 1_{L_1}$. This gives  $\phi^{-1}(b) = 1_{L_1}$, a contradiction. 
		
		Thus, $a$ have a complement in $L_2$ and hence  $L_2$ is a complemented lattice.
	\end{proof}
	
	By Corollary \ref{dist} and Theorem \ref{comp}, following corollary is straightforward.
	
	\begin{corollary}\label{boolean} 
		Let $L_1$ and $L_2$ be two lattices such that  $Com(L_1) \cong Com(L_2)$. Then $L_1$ is Boolean if and only if $L_2$ is Boolean.
	\end{corollary}
	
	\begin{theorem}\label{atomistic} 
		Let $L_1$ and $L_2$ be two lattices of finite length such that $Com(L_1) \cong Com(L_2)$. Then $L_1$ is atomistic and dual atomistic if and only if  $L_2$ is atomistic and dual atomistic.
\end{theorem}

\begin{proof}		
	   Let $\phi$ be a graph isomorphism from $Com(L_1)$ to $Com(L_2)$. Suppose that $L_1$ is atomistic and dual atomistic. We have to show that $L_2$ is also an atomistic and dual atomistic lattice. Without loss of generality, assume that $L_2$ is not an atomistic lattice. As $L_1$ is an atomistic lattice, every nonzero element that is not an atom is a join-reducible. Since $L_2$ is a lattice of finite length and not an atomistic lattice, there exists at least a nonzero join-irreducible element, which is not an atom,  say $y$. 
		
		\textbf{Case(1):} Suppose that element is $ y=1_{L_2} $. If $ 1_{L_2} $ is an atom, then nothing to prove. Suppose that $ 1_{L_2} $ is not an atom. Let $\phi(x)$ for some $x \in L_1$, be an element such that $\phi(x) \Yleft 1_{L_2}$. Clearly, $x < 1_{L_1}$. Since $L_1$ is an atomistic lattice, there exists an atom $p$ (say) in $L_1$ such that $p < 1_{L_1}$ and $p \nleq x$, i.e., $p \| x$ in $L_1$. So that $\phi(p) \| \phi(x)$ in $L_2$. Thus, as $\phi(x) \Yleft 1_{L_2}$, we have $\phi(p) \vee \phi(x) = 1_{L_2}$, a contradiction to the fact that $1_{L_2}$ is a join-irreducible.	
		
		\textbf{Case(2)}: Suppose that element $y$ is not $ 1_{L_2} $. Let $y=\phi(a) \neq 1_{L_2}$, for some $a \in L_1$, be such a join-irreducible element in $L_2$. Also, let $\phi(b)$ for some $b \in L_1 $ such that $\phi(b) \Yleft \phi(a)$. 
		
		We claim that every maximal chain containing $0_{L_2}$, $1_{L_2}$ and $\phi(a)$ must contains $\phi(b)$. 
		
		Suppose that there exists a maximal chain containing $0_{L_2}$, $1_{L_2}$ and $\phi(a)$ which do not contains $\phi(b)$. Clearly,  there is an element in that chain $\phi (d) \neq 0_{L_2}$ for some $d$ in $L_1$ such that $\phi (d) \Yleft \phi(a)$ (such a element exist because $L_2$ is of finite length and $\phi(a)$ is not an atom).
		
		As $\phi (b) \Yleft \phi(a)$ and $\phi (d) \Yleft \phi(a)$, we get  $\phi (b) \vee \phi (d) = \phi(a)$, a contradiction to the fact that $\phi(a)$ is a join-irreducible.  Hence, every maximal chain containing $0_{L_2}$, $1_{L_2}$ and $\phi(a)$ must contains $\phi(b)$. Thus, every maximal clique in $Com(L_2)$ containing $\phi(a)$ must contains $\phi(b)$. Since $Com(L_1) \cong Com(L_2)$, every maximal clique in $Com(L_1)$ containing $a$ must contains $b$, i.e., every maximal chain in $L_1$ containing $0_{L_1}$, $1_{L_1}$ and $a$ must contains $b$. 
		
		Now, since $\phi(b) < \phi(a)$ in $L_2$, i.e., $\phi(b) \sim \phi(a)$ in $Com(L_2)$, $b \sim a$ in $Com(L_1)$. If $b < a$, then since $L_1$ is atomistic, there exists an atom $p$ such that $p \nleq b$, i.e., $p \| b$  and $p < a$, that gives a maximal chain in $L_1$  containing $0_{L_1}, 1_{L_1}, a$ and $p$, which do not contains $b$, a contradiction. If $a < b$, then again as $L_1$ is dual atomistic, there exists a dual atom $q$ such that $b \nleq q$, i.e., $b \| q$ and $a < q$, which gives a maximal chain in $L_1$  containing $0_{L_1}, 1_{L_1}, a$ and $q$, which do not contains $b$, a contradiction.
		
		Thus, $L_2$ must be an atomistic lattice. Similarly, we can prove that $L_2$ is a dual atomistic lattice. 
	\end{proof}	
	
\section{Construction of non-isomorphic lattices with isomorphic comparability graphs}
	
	 Now we define an operation on two lattices to construct non-isomorphic lattices with isomorphic comparability graphs.
	 
	 Let $L_1$ and $L_2$ be two lattices and $L$ be the resultant lattice. We denote $L$ as $L_1 * L_2$. In the lattice $L$, $0_{L_2} = 1_{L_1}$ and $x < y$, for all $x \in L_1$, $y \in L_2$ and remaining all other compatibility relations in both $L_1$ and $L_2$ will be as it is. The following example illustrate it.
	
	
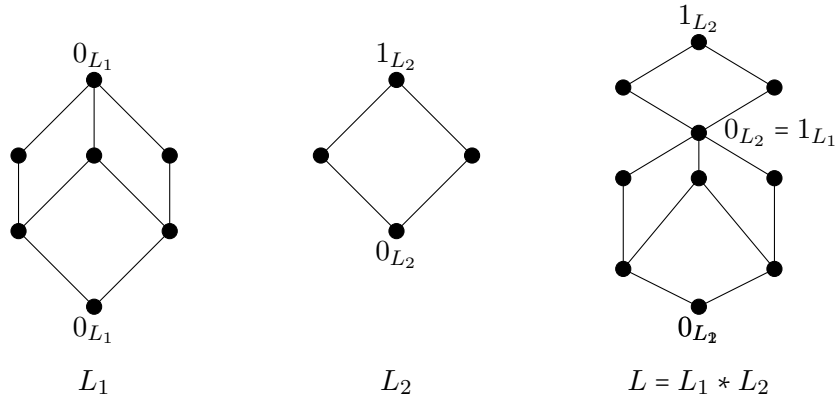
\begin{center}
		
		\begin{tikzpicture}[scale =1]
			
			\draw [fill=black] (0,-1) circle (.1); \draw [fill=black] (0,1) circle (.1); \draw [fill=black] (0,2) circle (.1);  \draw [fill=black] (1,0) circle (.1);  \draw [fill=black] (-1,0) circle (.1); \draw [fill=black] (1,1) circle (.1); \draw [fill=black] (-1,1) circle (.1);
			
			\draw (0,-1)--(1,0)--(1,1)--(0,2);
			
			\draw (0,-1)--(-1,0)--(-1,1)--(0,2);
			
			\draw	(-1,0)--(0,1)--(0,2);
			
			\draw
			(0,1)--(1,0);
			
			\draw node [below] at (0,-1) { $0_{L_1}$};
			\draw node [above] at (0,2) { $0_{L_1}$};
			\draw node [below] at (0,-1.7) { $L_1$};
			
			\begin{scope}[shift={(4,0)}]
				\draw [fill=black] (1,1) circle (.1);
				\draw [fill=black] (-1,1) circle (.1);
				\draw [fill=black] (0,2) circle (.1);
				\draw [fill=black] (0,0) circle (.1);
				
				\draw (0,0)--(1,1)--(0,2)--(-1,1)--(0,0);
				
				\draw node [below] at (0,0) { $0_{L_2}$};
				\draw node [above] at (0,2) { $1_{L_2}$};
				\draw node [below] at (0,-1.7) { $L_2$};
				
			\end{scope}
			
			\begin{scope}[shift={(8,0)}]
				\draw [fill=black] (0,-1) circle (.1);
				\draw [fill=black] (-1,-0.5) circle (.1);
				\draw [fill=black] (1,-0.5) circle (.1);
				\draw [fill=black] (-1,0.7) circle (.1);
				\draw [fill=black] (0,0.7) circle (.1);
				\draw [fill=black] (1,0.7) circle (.1);
				\draw [fill=black] (0,1.3) circle (.1);
				\draw [fill=black] (-1,1.9) circle (.1);
				\draw [fill=black] (0,2.5) circle (.1);
				\draw [fill=black] (1,1.9) circle (.1);
				
				\draw (0,-1)--(-1,-0.5)--(0,0.7)--(1,-0.5)--(0,-1);
				\draw (-1,-0.5)--(-1,0.7)--(0,1.3)--(-1,1.9)--(0,2.5)--(1,1.9)--(0,1.3)--(1,0.7)--(1,-0.5);
				\draw (0,1.3)--(0,0.7);
				
				\draw node [below] at (0,-1) { $0_{L_2}$};
				\draw node [above] at (0,2.5) { $1_{L_2}$};
				\draw node [below] at (0,-1) { $0_{L_1}$};
				\draw node [right] at (0.2,1.3) { $0_{L_2} = 1_{L_1}$};
				\draw node [below] at (0,-1.7) { $L = L_1 * L_2$};
			\end{scope}
		\end{tikzpicture}
	\end{center}
	
	Clearly, $0_{L_2} = 1_{L_1}$ is a dominating vertex in $Com(L)$. We will denote this vertex by `$v$'. (Here, Dominating vertex we mean, a vertex which is adjacent to remaining all other vertices in the graph). So, we have $Com(L) = Com(L_1) \vee Com(L_2) \vee \{v\}$.
	
	\begin{remark}\label{*operation}
		
		\begin{enumerate}
			\item Let $L^d$ denote the dual lattice of $L$. Then $Com(L) \cong Com({L_1}^d * L_2) \cong Com(L_1 * {L_2}^d ) \cong Com({L_1}^d * {L_2}^d)$
			
			\item This * operation is used  to generate non-isomorphic lattices having isomorphic comparability graphs.
			
			\item Let $G_1$ and $G_2$ be two comparability graphs of lattices $L_1, L_2$ then $G_1 \vee G_2 \vee \{v\}$ is also a comparability graph, where $v$ is a dominating vertex.
		\end{enumerate}
	\end{remark}

	\begin{theorem}
		If $G$ is a comparability graph of a  lattice with exactly one dominating vertex then we can write $G = G_1 \vee G_2 \vee \{v\}$, where $G_1$ and $G_2$ are comparability graphs of lattices and $v$ is a dominating vertex.
	\end{theorem}
	
	\begin{proof}
		Let $G$ be a comparability graph of a lattice that has exactly one dominating vertex, say $v$. Suppose that $L$ is a lattice such that $G \cong Com(L)$. Since $v$ is vertex, $v \nin \{0_L, 1_L\}$, i.e., $0_L < v < 1_L$. As $v$ is a dominating vertex, it is comparable with all elements of $L$.
		
		Let $A = \{x \in L / x < v\}$ and $B = \{x \in L / x > v\}$. Clearly, $A \bigcap B = \phi$ . We claim that $A \bigcup \{v\}$ and $B \bigcup \{v\}$ are lattices with respect to the same partial order as in $L$. Clearly, $0_L < v$ gives $0_L \in A$. Also, for all $x \in A$ we have $x < v$. Let $v = 1_A$. So, we have $0_L$ and  $v$ are the least and greatest elements of $A \bigcup \{v\}$ respectively. Suppose $a, b \in A \bigcup \{v\}$. Then $a, b \leq v$ and this implies $a \wedge b \leq v$ and $a \vee b \leq v$, i.e., $a \wedge b \in A 
		\bigcup \{v\}$ and $a \vee b \in A 
		\bigcup \{v\}$. Thus $A \bigcup \{v\}$ is a lattice. In fact it is a cover preserving sublattice of $L$. On the similar lines we can show that $B \bigcup \{v\}$ is a lattice. Let $L_1 = A \bigcup \{v\}$, $L_2 = B \bigcup \{v\}$ and $v = 1_{L_1} = 0_{L_2}$. Then $Com(L_1 * L_2) = G = Com(L_1) \vee Com(L_2) \vee \{v\}$, i.e., $G = G_1 \vee G_2 \vee \{v\}$, where  $G_1 = Com(L_1)$ and $G_2 = Com(L_2)$.
	\end{proof}
	
	Corollary \ref{modular} clearly shows that modularity is preserved under graph isomorphism. However, semi-modularity is not preserved under graph isomorphism. For that, let $L_1 = L_2 = S_7$. By Remark \ref{*operation}, we have $Com(L_1 * L_2) \cong Com({L_1}^d * L_2) \cong Com(L_1 * {L_2}^d ) \cong Com({L_1}^d * {L_2}^d)$. Since $S_7$ is an upper semimodular but not lower semimodular, $L_1 * L_2$ is an upper semimodular but not lower semimodular, $L_1^d * L_2$ is neither upper nor lower semimodular and  $L_1^d * L_2^d$ is lower semimodular but not upper semimodular. On the similar lines, using the * operation, we can have  examples for 0-distributivity, 1-distributivity and other similar properties.
	
	\section{Isomorphism Theorems}
	
	Until now, few lattice theoretical properties have been preserved under the graph isomorphism. But the Isomorphism Problem is still unanswered. The answer is positive in the case of modular lattices of length 3. Unfortunately, the class of modular lattices and moreover the class of distributive lattices do not give affirmative answers when the length of the lattices is greater than 3. The following theorems give us a partial solution to our main Isomorphism Problem for comparability graphs.
	
	\begin{remark}
		If a lattice $L$ is of length 3, then every nonzero nonunit element of $L$ is either an atom or a dual atom, i.e., every vertex in $Com(L)$ is either an atom or dual atom of $L$. Further, in the case of the modular lattice of length 3, every dual atom must contain an atom as the lattice is graded. So any two maximal chains must have the same height. 
	\end{remark}
	
	\begin{lemma}\label{length 3}
		Let $L_1$ and $L_2$ be two modular lattices of length 3 such that $Com(L_1) \cong Com(L_2)$. Then
		
		\begin{enumerate}
			\item If there exists an atom in $L_1$ that maps to an atom in $L_2$, then $At(L_1)$ maps to $At(L_2)$ and $DualAt(L_1)$ maps to $DualAt(L_2)$.
			
			\item If there exists an atom in $L_1$ that maps to a dual atom in $L_2$, then $At(L_1)$ maps to $DualAt(L_2)$ and $DualAt(L_1)$ maps to $At(L_2)$.
		\end{enumerate}
		
	\end{lemma}  
	
	\begin{proof}
		Let $\phi$ be a graph isomorphism from $ Com(L_1) $ to $ Com(L_2) $.
		
		\textbf{Proof of (1):} Suppose $a$ and $\phi (a)$ are atoms. We have to show that $At(L_1)$ maps to $At(L_2)$. Let $b$ be an atom of $L_1$ such that $a\neq b$. On the contrary, suppose that $\phi (b)$ is not an atom in $L_2$. Since $l(L_2)\equal 3$, $\phi(b)$ is a dual atom and as $a\nsim b$, $\phi(a) \nsim \phi(b)$, i.e., $\phi(a)$ and $\phi(b)$ are incomparable. This gives $\phi(a)\wedge \phi(b) \equal 0_{L_2}$ and $\phi(a)\vee \phi(b) \equal 1_{L_2}$. Since $L$ is modular and $\phi(b)$ is a dual atom, there exists an atom $\phi(c)$ such that $\phi(c) < \phi(b)$. Clearly, $\phi(c)\neq \phi(a)$ and $\phi(a)\wedge \phi(c) \equal 0_{L_2}$. 
		
		Suppose that $\phi(a)\vee \phi(c) \neq 1_{L_2}$. Let $\phi(a)\vee \phi(c) \equal \phi(d)$ for some $d\in L_1\setminus\{0_{L_1},1_{L_1}\}$. Clearly, $\phi(d)$ is a dual atom. Since $\phi(a)\sim \phi(d)$, $a\sim d$ and as $a$ is an atom, $d$ has to be a dual atom. Similarly, $\phi(c)\sim \phi(b)$ gives $c\sim b$ and as $b$ is an atom, $c$ has to be a dual atom. However, $\phi(c)\sim \phi(d)$ gives $c\sim d$, a contradiction to the fact that both are distinct dual atoms. Thus, we must have $\phi(a)\vee \phi(c) \equal 1_{L_2}$. 
		
		From the above discussion, $\{0_{L_2}, 1_{L_2}, \phi(a), \phi(b), \phi(c)\}$ forms a sublattice isomorphic to $N_5$, a contradiction. Hence $\phi(b)$ is an atom of $L_2$. This shows that $At(L_1)$ maps to $At(L_2)$. On the similar lines, we have $DualAt(L_1)$ maps to $DualAt(L_2)$.   
		
		\textbf{Proof of (2):} Suppose $a$ is an atom and $\phi (a)$ is a dual atom. We have to show that $At(L_1)$ maps to $DualAt(L_2)$. Let $b$ be an atom such that $a\neq b$. On the contrary, suppose $\phi (b)$ is not a dual atom. Since $l(L_2)\equal 3$, $\phi(b)$ is an atom and as $a\nsim b$, $\phi(a) \nsim \phi(b)$. This gives $\phi(a)\wedge \phi(b) \equal 0_{L_2}$ and $\phi(a)\vee \phi(b) \equal 1_{L_2}$. Since $\phi (b)$ is an atom, there exists a dual atom $\phi(c)$ for some $c \in L_1 \setminus \{0_{L_1},1_{L_1}\}$, such that $\phi(b) < \phi(c)$. As $\phi(a)$ and $\phi(c)$ are dual atoms, we have $\phi(a)\vee \phi(c)\equal 1_{L_2}$.
		
		Suppose $\phi(a)\wedge \phi(c)\neq 0_{L_2} $. Let $\phi(a)\wedge \phi(c)\equal \phi(d)$ for some $d\in L_1 \setminus \{0_{L_1},1_{L_1}\}$. Clearly, $\phi(d)$ is an atom because $\phi(d) < \phi(a)$ and $\phi(a)$ is a dual atom. Since $\phi(a)\sim \phi(d)$, $a\sim d$ and as $a$ is an atom, $d$ has to be a dual atom. Similarly, $\phi(c)\sim \phi(b)$ gives $c\sim b$ and as $b$ is an atom, $c$ has to be a dual atom. However, $\phi(c)\sim \phi(d)$ gives $c\sim d$, a contradiction to the fact that both are dual atoms. Thus, $\phi(a) \wedge \phi(c) \equal 0_{L_2}$.
		
		From the above discussion,  $ \{0_{L_2}, 1_{L_2}, \phi(a), \phi(b), \phi(c)\} $ forms a sublattice  isomorphic to $N_5$, a contradiction to the fact that $L_2$ is modular. Thus, $\phi(b)$ is a dual atom. This shows that $At(L_1)$ maps to $Dual At(L_2)$. On the similar lines, we have $DualAt(L_1)$ maps to $At(L_2)$.    
	\end{proof}

	\begin{theorem}\label{7}
		Let $L_1$ and $L_2$ be two modular lattices of length 3 such that $Com(L_1) \cong Com(L_2)$. Then either $L_1  \cong L_2$ or $L_1 \cong L_2^d$.
	\end{theorem}
	
	\begin{proof}	
		Let $\phi$ be a graph isomorphism from $ Com(L_1) $ to $ Com(L_2) $. Since each lattice is of length 3, every nonzero nonunit element is either an atom or a dual atom.
		
		\textbf{Case (1):} \textbf{Suppose that there exists an atom $c \in L_1$ such that $\phi(c)$ is an atom.} By Lemma \ref{length 3}, $At(L_1)$ maps to  $At(L_2)$ and $DualAt(L_1)$ maps to  $DualAt(L_2)$. Now, we define $\phi_e$ from $L_1$ to $L_2$ such that $\phi_e (0_{L_1})\equal 0_{L_2}$, $\phi_e (1_{L_1})\equal 1_{L_2}$ and $\phi_e (a)\equal \phi(a)$, for $a\notin \{0_{L_1},1_{L_1}\}$. It is clear that $\phi_e$ is a bijective map. We have to show that $\phi_e$ is a lattice isomorphism. For that, we will show that $\phi_e$ is a bi-order preserving map.
		
		Suppose $a < b$ for some $a, b \in L_1$. If $a\equal 0_{L_1}$, then $\phi_e(a)\equal \phi_e(0_{L_1}) = 0_{L_2}$ and hence $\phi_e(a)< \phi_e(b)$. If $b\equal 1_{L_1}$, then $\phi_e(b)\equal \phi_e(1_{L_1}) =  1_{L_2}$ and hence $\phi_e(a)< \phi_e(b)$. Now, let $a,b \nin \{0_{L_1}, 1_{L_1}\}$. Since $\ell(L) = 3$, $a$ is an atom and $b$ is a dual atom. By Lemma \ref{length 3}, $\phi_e(a)$ is an atom and $\phi_e(b)$ is a dual atom. As $a\sim b$, $\phi_e(a) \sim \phi_e(b)$. This gives $\phi_e(a) < \phi_e(b)$.
		
		Now, suppose $\phi_e(x) < \phi_e(y)$ for some  $x, y \in L_1$. If $\phi_e(x) = 0_{L_2} = \phi_e(0_{L_1})$, then $x=0_{L_1}$, as $\phi_e$ is bijective. This gives $x<y$. If $\phi_e(y)=1_{L_2}$, then $y=1_{L_1}$ gives $x<y$. Let $\phi_e(x), \phi_e(y) \nin \{0_{L_2},1_{L_2}\}$. Since $\ell(L_2)=3$, $\phi_e (x)$ is an atom and $\phi_e (y)$ is a dual atom. By Lemma \ref{length 3}, $x$ is an atom and $y$ is a dual atom. Since $\phi_e(x) \sim \phi_e(y)$, $x \sim y$. This gives $x < y$.
		
		Thus, $\phi_e$ is a bijective and bi-order preserving map. Hence $\phi_e$ is an isomorphism and $L_1\cong L_2$.

		\textbf{Case (2):} \textbf{Suppose that there exists an atom $c \in L_1$ such that it maps to a dual atom of $L_2$, i.e., $ c $ maps to an atom of $L_2^d$.} We know that modularity is a dual property for lattices. Since $L_2$ is a modular lattice, $L_2^d$ is also a modular lattice. Also, $Com(L_2)\cong Com({L_2}^d)$. This gives $Com(L_1)\cong Com({L_2}^d)$. Thus, we have $L_1$ and $L_2^d$ such that both are modular lattices of length 3 and $Com(L_1)\cong Com({L_2}^d)$. Since there exists an atom $c \in L_1$ that maps to an atom of $L_2^d$, we get the result by Case (1).	\end{proof}

	\begin{remark}
		Let $L$ be a lattice of length 3. Then $L$ is modular if and only if it is 0-modular. Hence in Lemma \ref{length 3} and Theorem \ref{7}, we can replace modularity by 0-modularity.
	\end{remark}

	As we discussed earlier, in the class of modular and distributive lattices, we do not get a lattice isomorphism from a graph isomorphism in the general case. Here is an example to justify it.
	
	\begin{example}\label{modularexample}
		In the following figure,  $L_1$ and $L_2$ are modular as well as distributive lattices with $Com(L_1) \cong Com(L_2)$, but neither $L_1 \ncong L_2$ nor $L_1 \ncong L_2^d$.   \end{example}
	
	\begin{center}
		\begin{tikzpicture}[scale =.5]
			
			\draw [fill=black] (0,0) circle (.1);
			\draw [fill=black] (1,1) circle (.1); 
			\draw [fill=black] (-1,1) circle (.1);  
			\draw [fill=black] (0,2) circle (.1);
			\draw [fill=black] (0,3) circle (.1);
			\draw [fill=black] (0,4) circle (.1);
			
			\draw (0,0)--(1,1)--(0,2)--(0,3)--(0,4);
			
			\draw (0,0)--(-1,1)--(0,2);
			
			\draw node [below] at (0,-0.8) { $L_1$};
			
			\draw node [below] at (0,-0.1) { $0$};
			
			\draw node [left] at (-1.1,1) { $b$};
			
			\draw node [right] at (1.1,1) { $c$};
			
			\draw node [left] at (-0.1,2) { $a$};
			
			\draw node [left] at (-0.1,3) { $d$};
			
			\draw node [above] at (0,4.1) { $1$};

			\begin{scope}[shift={(5,0)}]
				
				\draw [fill=black] (0,0) circle (.1);
				\draw [fill=black] (0,1) circle (.1); 
				\draw [fill=black] (-1,2) circle (.1);  
				\draw [fill=black] (1,2) circle (.1);
				\draw [fill=black] (0,3) circle (.1);
				\draw [fill=black] (0,4) circle (.1);
				
				\draw (0,0)--(0,1)--(-1,2)--(0,3)--(0,4);
				
				\draw (0,1)--(1,2)--(0,3);
				
				\draw node [below] at (0,-0.8) { $L_2$};
				
				\draw node [below] at (0,-0.1) { $0$};
				
				\draw node [left] at (-1.1,2) { $b_1$};
				
				\draw node [right] at (1.1,2) { $c_1$};
				
				\draw node [left] at (-0.1,1) { $a_1$};
				
				\draw node [left] at (-0.1,3) { $d_1$};
				
				\draw node [above] at (0,4.1) { $1$};
				
				
			\end{scope}
			
			\begin{scope}[shift={(12,0)}]
				
				\draw [fill=black] (-1,0) circle (.1);
				\draw [fill=black] (1,0) circle (.1);
				\draw [fill=black] (1,2) circle (.1);
				\draw [fill=black] (-1,2) circle (.1);
				
				\draw node [right] at (1,2) { $b$};
				\draw node [left] at (-1,0) { $d$};
				\draw node [left] at (-1,2) { $a$};
				\draw node [right] at (1,0) { $c$};
				
				\draw (1,2)--(-1,2)--(-1,0)--(1,2);
				\draw (-1,0)--(1,0)--(-1,2);
				
				\draw node [below] at (0,-0.8) { $Com(L_1) \cong Com(L_2)$};
				
			\end{scope}
			
		\end{tikzpicture}
		
	\end{center}
	
	\begin{remark}
		From Example \ref{modularexample}, it is clear that for modular lattices of length greater than 3, the comparability graph isomorphism does not give the lattice isomorphism. But in the case of covering graphs, Theorem \ref{jak1} and Theorem \ref{jak2} say that the covering graph isomorphisms are given by certain direct-product decomposition.
	\end{remark}
	
	In an atomistic and dual atomistic lattice, every nonzero element can be written as a join of some atoms, and every nonunit element can be written as a meet of some dual atoms. Also, there are some algebraic structures like vector spaces whose  lattice of substructures is atomistic and dual atomistic. 
	
	As we discussed earlier,  the atomisticity and dual atomisticity of a lattice is preserved under the graph isomorphism. Surprisingly, it gives a huge class in which graph isomorphism extends to lattice isomorphism. Let's discuss this interesting fact step by step.

	\begin{proposition}	
		In an atomistic or dual atomistic lattice, $0_L$ and $1_L$ are the only two elements comparable with all rest of the elements in $L$.	
	\end{proposition}
	
	\begin{proof}		
		Let $L$ be an atomistic lattice. On the contrary, suppose that there exists an element $a \nin \{0_L,1_L\}$ such that it is comparable with all other elements. Clearly, $a < 1_L$ and since $L$ is atomistic, there exists an atom  $p$ such that $p \nleq a$ and $p < 1_L$. By the assumption, $a$ is comparable with $p$ so that we must have $a < p$, a contradiction to the fact that $p$ is an atom, as $a \nin \{0_L, 1_L\}$. Thus, $0_L$ and $1_L$ are the only two elements comparable with all rest of the elements in $L$.
		
		The equivalent argument can be made when $L$ is dual atomistic.
	\end{proof}	
	
	Following lemmas describe the crucial role of atoms and dual atoms to get the lattice isomorphism from graph isomorphism.

	\begin{lemma} \label{atom}
		Let $L_1 $ and  $L_2$ be two atomistic and dual atomistic lattices and $Com(L_1) \cong Com(L_2)$. If $a \in L_1$ is an atom, then $\phi(a)$ is either an atom or a dual atom of $L_2$.
	\end{lemma}
	
	\begin{proof}		
		Let $\phi$ be a graph isomorphism from $ Com(L_1) $ to $ Com(L_2) $. Suppose $a \in L_1$ is an atom. We have to show that $\phi(a)$ is either an atom or a dual atom of $L_2$. On the contrary, suppose that $\phi(a)$ is neither an atom nor a dual atom of $L_2$. Let $\phi(p)$ be an atom of $L_2$ such that $\phi(p) < \phi(a)$, for some $p \in L_1$. Similarly,  $\phi(q)$ be a dual atom of $L_2$ such that $\phi(a) < \phi(q)$, for some $q \in L_1$. Since $L_2$ is an atomistic lattice and $\phi(p) < \phi(a)$, there exists an atom $\phi(p_1)$, for some $p_1 \in L_1$, different from $\phi(p)$ such that $\phi(p_1) < \phi(a)$. Similarly, as $L_2$ is a dual atomistic lattice, there exists a dual atom $\phi(q_1)$, for some $q_1 \in L_1$, different from $\phi(q)$ such that $\phi(a) < \phi(q_1)$.
		
		\par  From the above discussion, we have $a$ is adjacent to $p, p_1, q$ and $q_1$  in $Com(L_1)$. Since $a$ is an atom in $L_1$ so that the elements $p, p_1, q$ and $q_1$ are all greater than $a$. Also, as $\phi(p)$ and  $\phi(p_1)$ both are different atoms of $L_2$, they are not adjacent in $Com(L_2)$. This gives $p$ and $p_1$ are not adjacent in $Com(L_1)$, i.e., they are not comparable in $L_1$. Similarly, $\phi(q)$ and  $\phi(q_1)$ both are different dual atoms of $L_2$, they are not adjacent in $Com(L_2)$. This gives $q$ and $q_1$ are not adjacent in $Com(L_1)$, i.e., they are not comparable in $L_1$. Also, we have $\phi(p) < \phi(a) < \phi(q), \phi(q_1)$ and $\phi(p_1) < \phi(a) < \phi(q), \phi(q_1)$. So that $p \sim q, q_1$ and $p_1 \sim q, q_1$ in $Com(L_1)$. 
		
		Thus, we have either $q, q_1 < p$ or $q, q_1 > p$ and, either $q, q_1 < p_1$ or $q, q_1 > p_1$. Now, if either ($q, q_1 < p$ and  $q, q_1 > p_1$) or ($q, q_1 > p$ and $q, q_1 < p_1$), then $p$ and $p_1$ will be comparable, a contradiction. So that we have either $q, q_1 < p, p_1$ or $q, q_1 > p, p_1$. We consider the following cases.
		
		\textbf{Case(1)}: Suppose that $q, q_1 < p, p_1$. We have $\phi(p), \phi(p_1) < \phi(a) < \phi(q)$ in $L_2$. Since $L_2$ is atomistic, there exists another atom $\phi(p_2)$, for some $p_2 \in L_1$, such that $\phi(p_2) < \phi(q)$ and $\phi(p_2) \nleq \phi(a)$. Thus, $\phi(p_2) \sim \phi(q), \phi(p_2) \nsim \phi(a), \phi(p_2) \nsim \phi(p)$ and $\phi(p_2) \nsim \phi(p_1)$ in $Com(L_2)$, as $\phi(p_1)$ and $\phi(p_2)$ are atoms in $L_2$ . This gives $p_2 \sim q, p_2 \nsim a, p_2 \nsim p$ and $p_2 \nsim p_1$ in $Com(L_1)$. Since $p_2 \sim q$, we have either $p_2 < q$ or $p_2 > q$. If $p_2 < q$, then $p_2 < p, p_1$, as $q < p, p_1$, a contradiction to  $p_2 \nsim p$ and $p_2 \nsim p_1$. If $p_2 > q$, then $p_2 > a$, as $q > a$, a contradiction to  $p_2 \nsim a$.
		
		\textbf{Case(2)}: Suppose that $q, q_1 > p, p_1$. We have $\phi(q), \phi(q_1) > \phi(a) > \phi(p)$ in $L_2$. Since $L_2$ is dual atomistic, there exists another dual atom $\phi(q_2)$, for some $q_2 \in L_1$, such that $\phi(q_2) > \phi(p)$ and $\phi(q_2) \ngeq \phi(a)$. Hence,  $\phi(q_2) \sim \phi(p), \phi(q_2) \nsim \phi(a), \phi(q_2) \nsim \phi(q)$ and $\phi(q_2) \nsim \phi(q_1)$ in $Com(L_2)$. This gives $q_2 \sim p, q_2 \nsim a, q_2 \nsim q$ and $q_2 \nsim q_1$ in $Com(L_1)$. Since $q_2 \sim p$, we have either $q_2 < p$ or $q_2 > p$. If $q_2 < p$, then   $q_2 < q, q_1$, as $p < q, q_1$, a contradiction to  $q_2 \nsim q$ and $q_2 \nsim q_1$. If $q_2 > p$, then  $q_2 > a$, as $p > a$, a contradiction to  $q_2 \nsim a$.
		
		\par Thus, in both cases, we get contradictions. It means that our assumption, $\phi(a)$ is neither an atom nor a dual atom of $L_2$, was wrong. Hence  $\phi(a)$ must be either an atom or a dual atom of $L_2$.  
	\end{proof}	
	

	\begin{remark}
		If lattice $L$ is atomistic and dual atomistic, then $L^d$ is also atomistic and dual atomistic. Also, dual atoms of $L$ become atoms of $L^d$ and vice versa.
	\end{remark}
	
	From the duality principle, Lemma \ref{dualatom} follows immediately.
	
	\begin{lemma}\label{dualatom}
		Let $L_1 $ and  $L_2$ be two atomistic and dual atomistic lattices such that $Com(L_1) \cong Com(L_2)$. If $a \in L_1$ is a dual atom, then $\phi(a)$ is either an atom of $L_2$ or a dual atom of $L_2$.
	\end{lemma}
	
	\begin{lemma}\label{atom dual atom}
		
		Let $L_1 $ and  $L_2$ be two 0-modular, atomistic and dual atomistic lattices such that $Com(L_1) \cong Com(L_2)$. Then atoms of $L_1$, $At(L_1)$, maps to either atoms of $L_2$, $At(L_2)$, or dual atoms of $L_2$, $DualAt(L_2)$.
		
	\end{lemma}
	
	\begin{proof}		
		Let $\phi$ be a graph isomorphism from $Com(L_1)$ to $Com(L_2)$. Let $p_1$ and $p_2$ be two distinct atoms of $L_1$.  Suppose that $p_1 \vee p_2 = 1_{L_1}$. Since we are considering only connected graphs, there exist a path, $p_1-a_1-\dots-{p_2}$. As $p_1$ is an atom, $p_1 < a_1$. If $p_2 \sim a_1$, then we must have $p_2 < a_1$. This will give us $p_1 \vee p_2 \leq a_1$, which is not possible  as   $p_1 \vee p_2 = 1_{L_1}$.Thus, we must have $a_1 \| p_2$. Clearly, $a_1 \wedge p_2 = 0_{L_1}$. Now if $a_1 \vee p_2 \neq 1_{L_1}$, then again as $p_1 < a_1 < a_1 \vee p_2$ and $p_2 < a_1 \vee p_2$, we get $p_1 \vee p_2 \leq a_1 \vee p_2$, which is not possible as $p_1 \vee p_2 = 1_{L_1}$. Thus, $a_1 \vee p_2 = 1_{L_1}$. Clearly, the set $ \{p_1, a_1, p_2, 0_{L_1}, 1_{L_1}\} $ forms a sublattice containing $0_{L_1}$ and  isomorphic to  $N_5$, which is contradiction to the fact that $L_1$ is 0-modular.	Therefore we must have  $p_1 \vee p_2 \neq 1_{L_1}$. 
		
		In view of Lemma \ref{atom}, suppose $\phi(p_1)$ is an atom of $L_2$ and on the contrary, suppose that $\phi(p_2)$ is a dual atom of $L_2$. Clearly, $p_1, p_2 < p_1 \vee p_2$, i.e., $p_1 \sim p_1 \vee p_2$ and $p_2 \sim p_1 \vee p_2$ in $Com(L_1)$. So that we have $\phi(p_1) \sim \phi(p_1 \vee p_2)$ and $\phi(p_2) \sim \phi(p_1 \vee p_2)$ in $Com(L_2)$.
		
		\par Now, we have $\phi(p_1) \sim \phi(p_1 \vee p_2)$ and $\phi(p_2) \sim \phi(p_1 \vee p_2)$ in $Com(L_2)$. Since $\phi(p_1)$ is an atom of $L_2$, $\phi(p_1) < \phi(p_1 \vee p_2)$ and $\phi(p_2)$ is a dual atom of $L_2$ gives $\phi(p_2) > \phi(p_1 \vee p_2)$. So that we have $\phi(p_1) < \phi(p_1 \vee p_2) < \phi(p_2)$. This implies that $\phi(p_1) \sim \phi(p_2)$. However, this is not possible as $p_1$ and $p_2$  are distinct atoms, i.e., $p_1 \nsim p_2$. Thus, $\phi(p_2)$ must be an atom of $L_2$.
		
		On similar lines, we can prove that if $p_1$ maps to a dual atom of $L_2$, then $p_2$ also maps to a dual atom of $L_2$.
	\end{proof}	
	
	\begin{remark}
		In Lemma \ref{atom dual atom}, the 0-modularity is used to make the join of any two atoms nonunit and meet of any two dual atoms nonzero. Instead of 0-modularity, we can assume 1-modularity or join of any two atoms is nonunit and the meet of any two dual atoms is nonzero in the hypothesis.
	\end{remark}
	
	Now, we are ready to prove our main result.
	
	\begin{theorem}\label{iso}
		Let $L_1$ and $L_2$ be two 0-modular, atomistic and dual atomistic lattices. Then $Com(L_1) \cong Com(L_2)$ if and only if either $L_1 \cong L_2 $ or $L_1 \cong L_2^{d} $.
	\end{theorem}
	
	\begin{proof}		
		Let $ \phi $ be a graph isomorphism from $Com(L_1)$ to $Com(L_2)$. Clearly, $\phi $ is a bijective map from $L_1 \setminus \{0_{L_1 }, 1_{L_1 }\}$ to $L_2 \setminus \{0_{L_2}, 1_{L_2}\}$. Lemma \ref{atom dual atom} says, atoms of $L_1$ maps to either atoms of $L_2$ or dual atoms of $L_2$. We consider the following cases.
		
		\textbf{Case (1): Suppose atoms of $L_1$ maps to atoms of $L_2$.} 
		
		We define $ \phi_e $ from $L_1$ to $L_2$ such that  $\phi_e(0_{L_1 }) = 0_{L_2 } $, $\phi_e(1_{L_1 }) = 1_{L_2 } $ and $\phi_e(a) = \phi(a)$, for $ a \in L_1 \setminus \{0_{L_1}, 1_{L_1}\}$. 
		
		We claim that $L_1 \cong L_2$. For that, we will show that $\phi_e$ is a bijective and bi-order preserving map between $L_1$ and $L_2$.
		
		It is clear that $\phi_e$ is a bijective map. To show that $\phi_e$ is a bi-order preserving map, let $a, b \in L_1 $ such that $a < b$. If $a = 0_{L_1}$, then $ \phi_e(a) = \phi_e(0_{L_1}) = 0_{L_2}$ and $ 0_{L_2} < \phi(b)$. If $b = 1_{L_1} $, then $ \phi_e(b) = \phi_e(1_{L_1}) = 1_{L_2}$ and $ \phi(a) < 1_{L_2} $. Now, suppose $a, b \nin \{0_{L_1}, 1_{L_1}\} $. Since $a < b $, they are adjacent in $Com(L_1)$. As $Com(L_1) \cong Com(L_2)$, $\phi(a) \sim \phi(b)$ in $Com(L_2)$. Suppose that $ \phi(b) < \phi(a) $. As $ L_2 $ is an atomistic lattice and by the assumption, there exists an atom $\phi(p)$, for some atom $ p \in L_1 $, such that $ \phi(p) \leq \phi(a)$ and $ \phi(p) \nleq \phi(b)$, i.e., $ \phi(p) \sim \phi(a)$ and $ \phi(p) \nsim \phi(b)$ in $Com(L_2)$. Clearly if $  \phi(b) < \phi(p) $, then it will contradicts to the fact that $ \phi(p) $ is an atom  in $L_2$. This gives $ p \sim a $ and $ p \nsim b $ in $Com(L_1)$. As $p$ is an atom in $L_1$, we must have $ p \leq a $. However, $ a < b $ gives $ p < b $, i.e, $ p \sim b $, a contradiction to $ \phi(p) \nsim \phi(b) $. Thus, we must have $\phi(a) < \phi(b)$, i.e., $\phi_e (a) < \phi_e (b)$ and we are done.
		
		Let $ \phi_e (c) < \phi_e (d) $, for some $c, d \in L_1 $. If $ \phi_e (c) = 0_{L_2} $, then $ 0_{L_2} = \phi_e (0_{L_1})$ implies that $ c = 0_{L_1}$ and $ 0_{L_1} < d $. If $ \phi_e (d) = 1_{L_2} $, then $ 1_{L_2} = \phi_e (1_{L_1})$ implies that $ d = 1_{L_1}$ and $ c < 1_{L_1} $. Now, suppose $ \phi_e (c), \phi_e (d) \nin \{0_{L_2}, 1_{L_2}\}$. Since $\phi (c) < \phi (d)$, i.e., $\phi (c) \sim  \phi (d) $ in $ Com(L_2) $, $ c \sim d $ in $ Com(L_1) $. Suppose that $ d < c $. As $ L_1 $ is an atomistic lattice, there exists an atom $ p_1 \in L_1 $ such that  $ p_1 \leq c $ and $ p_1 \nleq d $, i.e., $ p_1 \sim c $ and $p_1 \nsim d $ in $ Com(L_1) $. If $  d < p_1 $, then it will contradicts to the fact that $ p_1 $ is an atom  in $ L_1 $ . This gives that $ \phi(p_1) \sim \phi(c) $ and $ \phi(p_1) \nsim \phi(d) $ in $ Com(L_2) $. As $p_1$ is an atom, by the assumption in this case, $ \phi(p_1) $ is an atom and hence $ \phi(p_1) \leq \phi(c) $. However, $ \phi(c) < \phi(d) $ which gives $ \phi(p_1) < \phi(d) $, i.e., $ \phi(p_1) \sim \phi(d) $, a contradiction. Hence, we must have $ c < d $. 
		
		Thus, $ \phi_e $ is a bijective and bi-order preserving map from $ L_1 $ to $ L_2 $ and hence $ L_1 \cong L_2 $. 
		
		\par \textbf{ Case (2): Suppose atoms of $ L_1 $ maps to dual atoms of $ L_2 $.} 
		
		Clearly, $Com(L_2) \cong Com(L_2^d)$. By the assumption, we have  $Com(L_1) \cong Com(L_2)$ so that $Com(L_1) \cong Com(L_2^d)$. Since $L_2$ is atomistic and dual atomistic, $L_2^d$ is also atomistic and dual atomistic. Moreover, atoms of $L_2$ become dual atoms $L_2^d$ and vice versa. Thus, by the assumption, in this case, atoms of $L_1$ map to atoms of $L_2^d$. Therefore, by using case(1), $L_1 \cong L_2^d$.   	
	\end{proof}	
	
	The proof of Theorem \ref{iso} uses both atomisticity and dual atomisticity of lattices. The following example shows that only atomisticity or dual atomisticity is insufficient.
	
	\begin{example}\label{Ex.atomistic}
		The lattice $L_1$, depicted in the following figure, is atomistic but not dual atomistic and lattice $L_2$ is neither atomistic nor dual atomistic still $Com(L_1) \cong Com(L_2)$.  
	\end{example}
	
	\begin{center}
		
		\begin{tikzpicture}[scale =.5]
			
			\draw [fill=black] (0,0) circle (.1);
			\draw [fill=black] (1,1) circle (.1); 
			\draw [fill=black] (-1,1) circle (.1);  
			\draw [fill=black] (-3,1) circle (.1);  
			\draw [fill=black] (3,1) circle (.1); 
			\draw [fill=black] (-2,2) circle (.1); 
			\draw [fill=black] (0,2) circle (.1);
			
			\draw [fill=black] (-1,3) circle (.1);
			\draw [fill=black] (0,4) circle (.1);
			
			\draw (0,0)--(-3,1)--(-2,2)--(-1,3)--(0,4);
			
			\draw (0,0)--(-1,1)--(-2,2);
			
			\draw
			(0,0)--(1,1)--(0,2)--(-1,3);
			
			\draw
			(0,2)--(-1,1);

			\draw
			(0,0)--(3,1)--(0,4);
			
			\draw node [below] at (0,-0.1) { $0$};

			\draw node [below] at (0,-0.8) { $L_1$};
			
			\draw node [below] at (-3,0.9) { $p_1$};
			
			\draw node [below] at (-1,0.9) { $p_2$};
			
			\draw node [below] at (1,0.9) { $p_3$};
			
			\draw node [below] at (3,0.9) { $p_4$};
			
			\draw node [below] at (-2,1.9) { $x$};
			
			\draw node [below] at (0,1.9) { $a$};

			\draw node [below] at (-1,2.9) {$a_1$};
			
			\draw node [above] at (0,4.1) { $1$};
			
			\begin{scope}[shift={(8,0)}]
				
				\draw [fill=black] (0,0) circle (.1); 
				\draw [fill=black] (-1,1) circle (.1); \draw [fill=black] (-3,2) circle (.1);  \draw [fill=black] (-1,2) circle (.1);  \draw [fill=black] (1,2) circle (.1); 
				\draw [fill=black] (-2,3) circle (.1); \draw [fill=black] (-1,1) circle (.1);
				\draw [fill=black] (0,3) circle (.1);
				\draw [fill=black] (0,4) circle (.1);
				\draw [fill=black] (3,2) circle (.1);

				\draw (0,0)--(-1,1)--(-3,2)--(-2,3)--(0,4);
				\draw (-1,1)--(1,2)--(0,3)--(0,4);
				\draw
				(-1,1)--(-1,2)--(-2,3);
				\draw
				(-1,1)--(-1,2)--(0,3);
				\draw (0,0)--(3,2)--(0,4);

				\draw node [below] at (0,-0.1) { $0$};
				
				\draw node [below] at (-1,0.9) { $a_1'$};
				
				\draw node [below] at (-3,1.9) { $p_1'$};
				
				\draw node [left] at (-1.1,2) { $p_2'$};
				
				\draw node [below] at (1,1.9) { $p_3'$};
				
				\draw node [below] at (3,1.9) { $p_4'$};
				
				\draw node [above] at (-2,2.9) { $x'$};
				
				\draw node [below] at (0,2.9) { $a'$};
				
				\draw node [above] at (0,4.1) {$1$};

				\draw node [below] at (0,-0.8) { $L_2$};
				
				
			\end{scope}
			
			\begin{scope}[shift={(16,0)}]
				
				\draw [fill=black] (1,1) circle (.1); 
				\draw [fill=black] (-1,1) circle (.1);  
				\draw [fill=black] (-3,1) circle (.1);  
				\draw [fill=black] (3,1) circle (.1); 
				\draw [fill=black] (-2,2) circle (.1); 
				\draw [fill=black] (0,2) circle (.1);
				
				\draw [fill=black] (-1,3) circle (.1);

				\draw (-3,1)--(-2,2)--(-1,3);
				
				\draw (-1,1)--(-2,2);
				
				\draw
				(1,1)--(0,2)--(-1,3);
				
				\draw
				(0,2)--(-1,1);
				
				\draw
				(-1,1)--(-1,3);
				
				\path[thin,black,draw](-3,1)..controls(-2.5,3)..(-1,3);
				
				\path[thin,black,draw](1,1)..controls(1,3)..(-1,3);		
				
				\draw node [below] at (0,-0.8) { $Com(L_1) \cong Com(L_2)$};
				
				\draw node [below] at (-3,0.9) { $p_1$};
				
				\draw node [below] at (-1,0.9) { $p_2$};
				
				\draw node [below] at (1,0.9) { $p_3$};
				
				\draw node [below] at (3,0.9) { $p_4$};
				
				\draw node [below] at (-2,1.9) { $x$};
				
				\draw node [below] at (0,1.9) { $a$};

				\draw node [above] at (-1,2.9) {$a_1$};
				
			\end{scope}
			
		\end{tikzpicture}
		
	\end{center}	
	
	\begin{remark}
		In view of Theorem \ref{duff}, Theorem \ref{iso} is stronger one.  Theorem \ref{duff} says that atomisticity with dual atomisticity is preserved under covering graph isomorphism, and the graph isomorphisms are given by certain direct product decompositions. However, Theorem \ref{iso} states that lattice isomorphisms directly give the graph isomorphisms.
		
		It is a natural question whether the 0-modularity becomes equivalent with atomisticity with dual atomisticity in the case of length 3. The answer is negative. If we take a chain of length 3, it is 0-modular but neither atomistic nor dual atomistic. Also, consider the lattice $L$ depicted in the following figure. Clearly, it is an atomistic and dual atomistic lattice but not 0-modular.
	\end{remark}
	
	
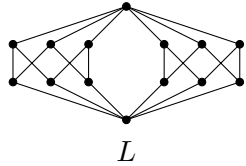
\begin{center}
		\begin{tikzpicture}[scale =.5]
			
			\draw [fill=black] (0,0) circle (.1);
			\draw [fill=black] (-1,1) circle (.1); 
			\draw [fill=black] (-2,1) circle (.1);  
			\draw [fill=black] (-3,1) circle (.1);
			\draw [fill=black] (1,1) circle (.1); 
			\draw [fill=black] (2,1) circle (.1);  
			\draw [fill=black] (3,1) circle (.1);
			\draw [fill=black] (0,3) circle (.1);
			\draw [fill=black] (-1,2) circle (.1); 
			\draw [fill=black] (-2,2) circle (.1);  
			\draw [fill=black] (-3,2) circle (.1);
			\draw [fill=black] (1,2) circle (.1); 
			\draw [fill=black] (2,2) circle (.1);  
			\draw [fill=black] (3,2) circle (.1);

			\draw (0,0)--(-1,1)--(-1,2)--(-2,1)--(-3,2)--(-3,1)--(-2,2)--(-1,1);
			
			\draw (0,0)--(1,1)--(1,2)--(2,1)--(3,2)--(3,1)--(2,2)--(1,1);
			
			\draw (0,0)--(2,1);
			\draw (0,0)--(3,1);
			\draw (0,0)--(-2,1);
			\draw (0,0)--(-3,1);
			
			\draw (0,3)--(1,2);
			\draw (0,3)--(2,2);
			\draw (0,3)--(3,2);
			\draw (0,3)--(-1,2);
			\draw (0,3)--(-2,2);
			\draw (0,3)--(-3,2);

			\draw node [below] at (0,-0.3) { $L$};
			
		\end{tikzpicture}
		
	\end{center}

	We close the paper by applying Theorem \ref{iso} to the subspace inclusion graph of a finite dimensional vector space. 
	
	\begin{lemma}\label{vector}
		Let $V_1$ and $V_2$ be two finite dimensional vector spaces over the same field $F$. Then $ V_1 \cong V_2 $ if and only if $ L(V_1) \cong L(V_2) $.
	\end{lemma}
	
	\begin{proof}
		Since $V_1$ and $V_2$ both are finite dimensional vector spaces, it is sufficient show that $dim(V_1) = dim(V_2)$. Since $ L(V_1) \cong L(V_2) $, we have  $l(L(V_1))$ = $l(L(V_2)) $. However, the length of the subspace lattice of a finite-dimensional vector space is the dimension of that vector space. Thus, $ dim(V_1) = dim(V_2) $ and hence $ V_1 \cong V_2 $.
		
		The converse part is quite clear.
	\end{proof}
	
	\begin{corollary}[{\cite[Corollary 5.2]{ad}}]\label{vspcorrolary}
		Let $V_1$ and $V_2$ be two finite dimensional vector spaces over the same field $F$. Then $V_1$ and $V_2$ are isomorphic as vector spaces if and only if $In(V_1)$ and $In(V_2)$ are isomorphic as graphs. 
	\end{corollary}	
	
	\begin{proof}		
		Suppose $In(V_1)$ is isomorphic to $In(V_2)$, i.e., $Com(L(V_1)) \cong Com(L(V_2))$. Since $ L(V_1) $ and $ L(V_2) $ both are bounded, atomistic and dual atomistic lattices, by Theorem \ref{iso}, we get either $ L(V_1) \cong L(V_2) $ or $ L(V_1) \cong L(V_2)^d$. However, $ L(V_2) $ is self dual, i.e., $ L(V_2) \cong L(V_2)^d $. Thus, by Lemma \ref{vector}, we get $V_1 \cong V_2$.
		
		The converse part is quite clear.
	\end{proof}

	\vskip 10pt 
	
\noindent \textbf{Conclusion:} In this paper, we have discussed about the lattice theoretical properties that are preserved under the graph isomorphism like modularity, distributivity, etc. Along with this we obtained nice relation between chordal graphs, distance-hereditary graphs and dismantlable lattices. Along with that we developed a technique to construct non-isomorphic lattices with isomorphic comparability graphs. In order to find the answer for the Isomorphism Problem, the lattice theoretic properties which are preserved under graph isomorphisms are crucial. The Isomorphism Problem is one of the important problem in the field of graph theory. We find two classes of lattices in which the the graph isomorphism gives the lattice isomorphism. Those classes are modular lattices of length 3 and $0$-modular lattices with atomisticity and dual atomosticity. As the subspace lattice of finite dimensional vector space is $0$-modular, atomistic and dual atomistic, we applied the results to the subspace inclusion graph of a vector space.

\vskip 10pt 
\noindent\textbf{Conflict of interest:} The authors declare that there is no conflict of interest regarding the publishing of this paper.
\vskip 10pt 
\noindent\textbf{Authorship Contributions:} Both the authors contributed equally to the study of comparability graphs of lattices and their applications to the inclusion graphs of vector spaces. Both the authors read and approved the final version of the manuscript.

\vskip5pt 
\noindent\textbf{Data Availability Statement:} Data sharing is not applicable to this article as no datasets were generated or analyzed during the current study.

\end{document}